\documentclass[12pt]{article}
\usepackage{amsmath,amsthm,amsfonts,amssymb,amscd}
\usepackage[usenames]{color}
\usepackage[dvips]{graphicx}
\usepackage{epsfig}
\def\phi{\varphi}
\def \P {\mathbb P}

\begin{document}

\newtheorem{nt}{Notation}
\newtheorem{cor}{Corollary}
\newtheorem{lem}{Lemma}
\newtheorem{predl}{Proposition}
\newtheorem{theorem}{Theorem}
\newtheorem{defin}{Definition}
\newtheorem{zam}{Remark}

\bigskip
\centerline{\bf Efficiency of Exponentiality Tests Based }
\centerline{\bf on a Special Property of Exponential Distribution}

\bigskip

\centerline{Nikitin Ya. Yu.$^{\diamond,\dagger,}$\footnote{Corresponding author, e-mail {\it yanikit47@gmail.com}},\, 
and Volkova K. Yu.$^{\diamond}$}

\bigskip

\centerline{\it $^\diamond$  Saint-Petersburg State University, 7/9 Universitetskaya nab., }
\centerline{ \it Saint-Petersburg, 199034 Russia.}

\medskip

\centerline{\it $^\dagger$ National Research University - Higher School of Economics,}
\centerline{\it  Souza Pechatnikov, 16, St.Peters\-burg 190008, Russia}

\bigskip

\begin{abstract}
New goodness-of-fit tests for exponentiality based on a particular property of exponential law are constructed.
Test statistics are functionals of $U$-empirical processes. The first of these statistics is of integral type, the second one is a Kolmogorov type statistic. We show that the kernels corresponding to our statistics are non-degenerate. The limiting distributions and large deviations  of new statistics under the null hypothesis are described.
Their local Bahadur efficiency for various parametric alternatives is calculated and is compared with simulated powers of new tests. Conditions of local optimality of new statistics in Bahadur sense are discussed and examples of "most favorable" alternatives are given. New tests are applied to reject the hypothesis of exponentiality for the length
of reigns of Roman emperors which was intensively discussed in recent years.
\end{abstract}

\vskip10pt


\vskip10pt


\vskip25pt

\section{Introduction}
\bigskip

\indent The general problem of exponentiality testing is stated as follows. Let $X_1,\ldots,X_n$ be
nonnegative independent observations  having a continuous distribution
function (df) $F$ and a density $f.$ We wish to test
the composite null-hypothesis $ H_0: F(x)$ is a df of an exponential law with density  $f(x)=\lambda e^{-\lambda x}, x \geq 0,$
where $ \lambda > 0$ is an unknown scale parameter, against the following alternative: $F$ is a df of a nonexponential law.

There exist numerous tests of exponentiality based on various ideas \cite{Ahsan}, \cite{A90}, \cite{BB95}, \cite{DY}, \cite{HM05}, \cite{NJ02}.
Among them a good few tests are based on characterizations. This is a relatively fresh idea which manifests growing popularity in goodness-of-fit testing,
and in particular, in exponentiality testing, see, e.g., \cite{Ang},  \cite{BarHen}, \cite{HM02}, \cite{JansV}, \cite{Jovan}, \cite{Milo}, \cite{Niknik}, \cite{NikVol}, \cite{NouArg}, \cite{Volk}.

Recently Noughabi and Arghami proved and used in \cite[Theor.1]{NouArg}  the following "characterization" of
exponential law for testing of exponentiality:\\
\indent {\it Let $X_1,X_2$ be two independent identically distributed nonnegative rv's having a continuous df $F.$ Then $Y=X_1/X_2 $ has the df $F_{(2,2)}$ if and only if $F$ is exponential. Here
$F_{(2,2)}$ is the df of Fisher distribution with 2 and 2 degrees of freedom so that
$$
F_{(2,2)}(y)=\frac{y}{1+y},\, \, y \geq 0.
$$}
 
 In fact this property  {\it is not} the proper  characterization of exponential law. This is known since  the paper of Kotlarski \cite{Kotla} which was preceded by the work of Mauldon \cite{Mau}. In particular, Kotlarski gave three examples of {\it non-exponential} densities for $X_1$ and $X_2$ under which the distribution of $Y$ is {\it still} $F_{2,2}$. These three densities are
$$
\lambda x^{-2} \exp(-\lambda x^{-1}){\bf 1}\{x>0\},\quad (1+x^2)^{-\frac32}{\bf 1}\{x>0\},\quad  \mbox{and} \quad  x(1+x^2)^{-\frac32}{\bf 1}\{x>0\}.
$$
Presumably Noughabi and Arghami got the erroneous result because of inaccurate application of the characterization
result of Kotz and Steutel \cite{Kotz}. The same concerns item iii) of their Theorem 1 in \cite{NouArg}.

However, one can build the statistical tests based on properties of distributions which are not the proper characterizations as well. Of course, this will lead to inconsistency of such tests against {\it certain} alternatives. But many famous tests well known in statistical practice are inconsistent against certain {\it special} alternatives, for instance, the chi-square test, the Wilcoxon test (and many other rank tests), the Gini test, and even the likelihood ratio test.

Moreover, according to the usual concepts of testing statistical hypotheses, the evidence can be sufficient only for the rejection of the null-hypothesis $H_0$. On the contrary, its definitive acceptance is hardly possible but any new test "failing to reject"  $H_0$ gradually brings the statistician to the perception of the validity of $H_0$.

The aim of the present paper is to test the hypothesis $H_0$ using the same property of exponential law as used in \cite{NouArg} and formulated above. We will construct two test statistics which turn out to be quite sensitive and efficient. We justify it by calculation of their local Bahadur efficiency against common alternatives and by simulation of their power.

Consider instead of the standard empirical df
$$F_n(t)=n^{-1}\sum_{i=1}^n\textbf{1}\{X_i<t\}, t \geq 0,
$$
the \, $U$-empirical df
\begin{eqnarray*}
H_n(t)=\frac{1}{n(n-1)}\sum\limits_{1\leq i< j
\leq n}\left(\textbf{1}\left\{\frac{X_{i}}{X_{j}}<t\right\}+\textbf{1}\left\{\frac{X_{j}}{X_{i}}<t\right\}\right).
\end{eqnarray*}

It is known that the properties of $U$-empirical df's are similar
to the properties of usual empirical df's, see \cite{HJS}, \cite{Jan}. Hence for large $n$ and under $H_0$ the df $H_n$ should be close to Fisher's df $F_{(2,2)},$
 and we can measure their closeness using some test statistics.

We suggest two scale-invariant statistics
\begin{align}
W_n&=\int_{0}^{\infty} \left(\frac{t}{1+t}-H_n(t)\right)\mu e^{-\mu t }dt, \mu >0 ,\label{W_n}\\
D_n&=\sup_{t \geq 0}\mid \frac{t}{1+t}-H_n(t)\mid\label{D_n},
\end{align}
assuming that their large absolute values are critical. We have inserted the exponential weight with some indefinite value of $\mu>0$ under the sign of integral in order to guarantee its convergence but for brevity we omit $\mu$ in the notation of statistic.

We discuss the limiting distributions of these statistics under the null hypothesis and calculate
their efficiencies against common alternatives. We use the notion of local exact Bahadur efficiency (BE)
\cite{Bahadur}, \cite{Nik},
as the statistic $D_n$ has the nonnormal limiting distribution, and
hence the Pitman approach to the calculation of efficiency is not applicable.
However, it is known that the local BE and the limiting Pitman efficiency
usually coincide, see  \cite{Wie}, \cite{Nik}.

The  large deviation asymptotics is the key tool for the evaluation of the
exact BE, and we address this question using the results of \cite{nikiponi} and \cite{Niki10}.
Finally, we study the conditions of local optimality of our tests and describe the
"most favorable" alternatives for them.

We present the simulated powers of new tests and enlarge the paper by the example of application to real data. Namely, as an application of new exponentiality tests, we examine the interesting question on the durations of reigns for Roman emperors discussed by Khmaladze and his coauthors \cite{Khmal}, \cite{Khm}. Our tests firmly reject the hypothesis of exponentiality, and this contradicts the findings of Khmaladze and his team,  see also \cite{Keague} and \cite{Pisk}.

We stress that usually in the papers on testing based on characterizations one uses the equality in distribution
of two statistics $T_1$ and $T_2$:
$$
T_1(X_1,\dots,X_k)\, \stackrel{d}{=} \, T_2(X_1,\dots,X_m)
$$
which characterizes the family of distributions or some specific property, e.g., symmetry of distribution. But in our paper we use a {\it different relation} when a certain statistic {\it has the prescribed distribution}, and this characterizes or  {\it strongly restraints} the distribution of the sample. It seems probable that other tests of fit can be build on the ground of this apparently new approach.

\section{ Integral statistic $W_{n}.$}

 \subsection{Limiting properties of statistic $W_{n}.$}

\indent The statistic $W_{n}$ is exactly the $U$-statistic of degree  2  with
the centered kernel
\begin{gather*}
\Phi(X,Y)=
1- \mu e^{\mu}E_1(\mu)- \frac{1}{2}e^{-\mu \frac{X}{Y}}-\frac{1}{2}e^{-\mu \frac{Y}{X}} ,
\end{gather*}
where $$
E_1 (\mu)= Ei(1,\mu)=\int_{1}^{\infty}e^{-\mu t}t^{-1}dt, \ Re \ \mu > 0 ,
$$
is the exponential integral, see \cite[Ch.5]{Abr}.

Let $X, Y$ be independent rv's from the standard exponential distribution. To prove that
 the kernel $\Phi(X, Y)$ is non-degenerate, let us calculate its projection $\phi_{\mu}(s).$ For a fixed $X=s, \, s\geq 0$ we have:
\begin{gather*}
\phi_\mu(s) := E(\Phi(X, Y)\mid X=s) =1- \mu e^{\mu}E_1(\mu)- \frac12 Ee^{-\mu \frac{s}{Y}}-\frac12 Ee^{-\mu \frac{Y}{s}}.
\end{gather*}

After some computations we find that the projection $\phi_\mu (s)$ is equal to
$$
\phi_\mu (s) = 1- \mu e^{\mu}E_1(\mu)- \sqrt{\mu s}\,K_1( 2 \sqrt{\mu s})- \frac{s}{2(\mu+s)}.\label{phi}
$$
where $K_1(y)$ is the modified Bessel function of the second kind.

The mean of this projection is equal to zero. Its variance under $H_{0}$ and for arbitrary value of $\mu>0$ equals $\Delta_W^2 (\mu) = E\phi_\mu^2(X),$
 it is positive  and can be obtained using  numerical methods (see Fig. 1), according to the formula
\begin{gather*}
\Delta_W^2(\mu) = \int_{0}^{\infty} \phi_\mu^2 (s) e^{-s} ds .
\end{gather*}

\begin{figure}[h!]
\begin{center}
\includegraphics[scale=0.3]{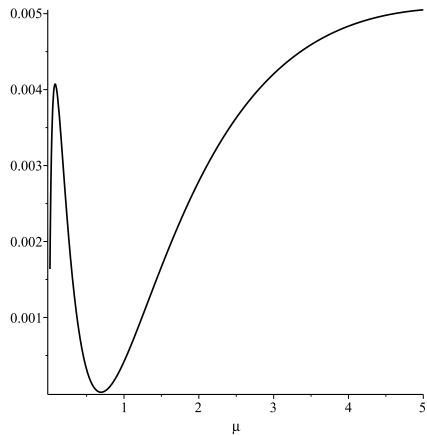}\caption{Plot of the function $\Delta_W^2(\mu)$. }
\end{center}
\end{figure}

Therefore the kernel $\Phi$ is centered and non-degenerate.
We can apply Hoeffding's theorem on asymptotic normality of $U$-statistics, see \cite{Hoeffding}, \cite{Kor},
which implies the following result:
\begin{theorem}
Under null hypothesis as $n \rightarrow \infty$ the statistic $\sqrt{n}W_{n}$ is asymptotically normal
so that
$$\sqrt{n}W_{n} \stackrel{d}{\longrightarrow}{\cal
{N}}(0,4\Delta_W^2(\mu)).$$
\end{theorem}

The (logarithmic) large deviation asymptotics of the sequence of statistics $W_{n}$ under $H_0$ follows from the following
result. It was derived using the theorem on large deviations (see again \cite{nikiponi},
 \cite{Niki10}), applied to the centered, bounded and non-degenerate kernel $\Phi.$
\begin{theorem}
For  $a>0$ under $H_0$ one has
$$
\lim_{n\to \infty} n^{-1} \ln \P ( W_n >a) = - f_W(a),
$$
where the function $f_W$ is continuous for sufficiently small $a>0,$ and $$
f_W(a)  = \frac{a^2}{8 \Delta_W^2(\mu)}(1 + o(1))  , \, \mbox{as } \, a \to 0.
$$
\end{theorem}

\subsection{Some notions from Bahadur theory}

\indent Suppose that under the alternative $H_1$ the observations have the df $G(\cdot,\theta)$ and the density $g(\cdot,\theta), \ \theta \geq 0,$ such that
$G(\cdot, 0)$ is the exponential df with some scale parameter.
The measure of Bahadur efficiency (BE) for any sequence $\{T_n\}$ of test statistics is the exact slope
$c_{T}(\theta)$ describing the rate of exponential decrease for the
attained level under the alternative df $G(\cdot,\theta).$ According to Bahadur theory  \cite{Bahadur}, \cite{Nik} the exact slopes may be found by
using the following Proposition.

\noindent {\bf Proposition}. \,{\it Suppose that two following
conditions hold:
\[
\hspace*{-3.5cm} \mbox{a)}\qquad  T_n \
\stackrel{\mbox{\scriptsize $\P_\theta$}}{\longrightarrow} \
b(\theta),\qquad \theta > 0,\nonumber \] where $-\infty <
b(\theta) < \infty$, and $\stackrel{\mbox{\scriptsize
$\P_\theta$}}{\longrightarrow}$ denotes convergence in probability
under $G(\cdot\ ; \theta)$;
\[
\hspace*{-2cm} \mbox{b)} \qquad \lim_{n\to\infty} n^{-1} \ \ln \
\P_{H_0} \left( T_n \ge t \ \right) \ = \ - f_T(t)\nonumber
\] for any $t$ in an open interval $I,$ on which $f_T$ is
continuous and $\{b(\theta), \: \theta > 0\}\subset I$. Then
$$c_T(\theta) \ = \ 2 \ f_T(b(\theta)).$$}

We have already found the large deviation asymptotics necessary for b). In order to evaluate the exact slope it remains to verify the condition a) of this Proposition which represents some form of the Law of Large Numbers under the alternative.

Note that the exact slopes for any $\theta$ satisfy the inequality (see \cite{Bahadur}, \cite{Nik})
\begin{equation}
\label{Ragav}
c_T(\theta) \leq 2 KL(\theta),
\end{equation}
where $KL(\theta)$ is the Kullback-Leibler "distance"\, between the alternative and the null-hypothesis $H_0.$
In our case $H_0$ is composite, hence for any alternative density $g_j(x,\theta)$ one has
\begin{equation}\label{kul}
KL_j(\theta) = \inf_{\lambda>0} \int_0^{\infty} \ln \big[g_j(x,\theta) / \lambda \exp(-\lambda x) \big] g_j(x,\theta) \ dx.
\end{equation}
This quantity can be easily calculated as $\theta \to 0$ for particular alternatives.
According to (\ref{Ragav}), the local BE of the sequence of statistics ${T_n}$ is defined as
$$
e^B (T) = \lim_{\theta \to 0} \frac{c_T(\theta)}{2KL(\theta)}.
$$

\bigskip

\subsection{Local Bahadur efficiency of $W_n.$}

\begin{nt}\label{nt}
Denote by $\cal G$
the class of densities $ g(\cdot \ ,\theta)$ with the df's $G(\cdot \ ,\theta), \theta \ge 0,$ which satisfy the regularity conditions from
\cite[Ch.6]{Nik} with possibility to differentiate with respect to $\theta$ under the integral sign in all appearing integrals.
\end{nt}

We present the following alternatives against exponentiality which will be considered for both tests
 in this paper:

\begin{enumerate}
\item[i)] Weibull distribution with the density
$$g_1(x,\theta)=(1+\theta)x^\theta \exp(-x^{1+\theta}),\theta \geq 0, x\geq 0;$$

\item[ii)] Gamma distribution with the density
$$g_2(x,\theta)=\frac{x^{\theta}}{\Gamma(\theta+1)}e^{-x}, \theta \geq 0, x\geq 0;$$

\item[iii)] exponential mixture with negative weights (EMNW($\beta$)) (see \cite{vjevremovic})
$$ g_3(x)=(1+\theta)e^{-x}-\theta\beta e^{-\beta x}, \theta\in \Big[0,\frac{1}{\beta-1}\Big], \beta >1, x\geq 0;$$

\item[iv)] exponential distribution with the resilience parameter, or the Verhulst distribution (see \cite[p.333]{MO})   with the density
$$
g_4(x,\theta) = (1+\theta)\exp(-x)(1 - \exp(-x))^{\theta}, \theta \geq 0, x \geq 0.
$$
\end{enumerate}

From \eqref{kul} one can find the Kullback-Leibler "distance" for each alternative as $\theta \to 0$:
\begin{gather}\label{k}
KL_1(\theta) \sim \frac{\pi^2}{12}\theta^2;\quad
KL_2(\theta) \sim \left(\frac{\pi^2}{12}-\frac12 \right) \theta^2;\notag \\
KL_3(\theta) \sim \frac{(\beta-1)^4}{2\beta^2(2\beta-1)}\theta^2; \quad
KL_4(\theta) \sim \left(\frac{\pi^2}{6}-\frac{\pi^4}{72}\right)\theta^2.
\end{gather}

For statistic $W_n$ we can derive the following asymptotics as $\theta \to 0$ from \cite{NiPe}.
\begin{lem}\label{b_W}
For a given alternative density $g(x,\theta)$ from the class $\cal G$ (see Notation~\ref{nt}) under condition $\int_0^{\infty} | g''_{\theta \theta}(x,0)| dx < \infty$ we get
\begin{gather*}
b_W(\theta)\sim 2\theta  \int_{0}^{\infty} \phi_{\mu}(x)h(x)dx, \text{ where }
h(x)=g'_{\theta}(x,0).
\end{gather*}
\end{lem}

We take $\mu=2$ for definiteness in the exponential weight $\mu e^{-\mu t}$,   so for this case the variance is $\Delta_W^2(2)= 0.0028.$
Using \eqref{k} we gather in Table \ref{fig: tab2} the values of function $b_W(\theta),$ local exact slopes  as $\theta \to 0$ and local BE
for statistics $W_n.$ In the case of the third alternative EMNW we take the value $\beta =3$ as in the recent paper \cite{Milo}. All this was obtained using the MAPLE package.
\begin{table}[!hhh]\centering
\caption{Local Bahadur efficiency for $W_n, \mu=2$ with $\Delta_W^2(\mu)= 0.0028.$}
\bigskip
\begin{tabular}{|c|c|c|c|}
\hline
Alternative &  $b_W(\theta)$& $c_W(\theta)$& Efficiency\\
\hline
Weibull &    0.123 $\theta$ & 1.357 $\theta^2$  & 0.825 \\
Gamma &  0.081 $\theta$ & 0.590  $\theta^2$  & 0.915  \\
EMNW ($\beta =3$)  &     0.056 $\theta$ & 0.284 $\theta^2$ &  0.800\\
Verhulst& 0.078 $\theta$      &0.541 $\theta^2$&0.927\\
\hline
\end{tabular}
\label{fig: tab2}
\end{table}
We observe here {\it remarkably high} values of local BE for common alternatives.

In Table \ref{fig: tabsimw} we present the simulated powers for our alternatives when  $\mu=2.$ The simulations have been performed for $n=100$ with 10,000 replicates
for the appropriate significance level $\alpha.$

\begin{table}[htbp]
\centering
\bigskip
\caption{Simulated powers for statistic $|W_n|, \mu=2.$ }
\bigskip
\centering
\begin{tabular}{|c|c|ccc|}
  \hline
 Alternative & $\theta$ &  $\alpha=0.1$ & $\alpha=0.05$ & $\alpha=0.01$ \\
  \hline
 Weibull&0.5 &  0.999 &0.997 &0.985  \\
   &0.25 & 0.822 & 0.717 & 0.499 \\
 \hline
  Gamma&0.5 & 0.922& 0.856 & 0.669   \\
   &0.25 & 0.506 & 0.366 & 0.186   \\
 \hline
  EMNV ($\beta=3$) &0.5 & 0.997 & 0.992 &0.959   \\
   &0.25 & 0.513 & 0.378& 0.193   \\
 \hline
 Verhulst&0.5 & 0.890 & 0.804 &0.600\\
   &0.25 & 0.467 & 0.333& 0.161 \\
 \hline
 \end{tabular}
\label{fig: tabsimw}
\end{table}

Note that there is no theoretical reasons for closeness of local efficiencies to the powers. However, if we take, for instance, the "realistic" values $\theta =0.5$ and $\alpha =0.05,$ then the ordering of tests is similar under both criteria. At the same time, the local BE under Weibull and Gamma alternatives has been calculated for many tests of exponentiality, see, e.g., \cite{Niknik}, \cite{NikTchir}, \cite{NikVol}, \cite{Pisk}, \cite{NikiVolki}, \cite{NTch}, \cite{Milo}.
It can be supposed that our test statistic $W_n$ is probably  more efficient than the tests considered in these papers.
So we may hope that the new test based on $W_n$ is able to reject the exponentiality hypothesis when the other tests are unfit for it. See section 4 below for partial confirmation of this.

 \section{ Kolmogorov-type statistic $D_{n}$}

\indent Now we consider the Kolmogorov type statistic (\ref{D_n}).
For fixed  $t$ the difference $\frac{t}{1+t}-H_n(t)$
is a family of $U$-statistics with the kernel $\Xi$ depending on $t \geq 0:$
\begin{gather*}
\Xi(X_{i},X_{j}; t)=\frac{t}{1+t}-
\frac{1}{2}\textbf{1}\{\frac{X_{i}}{X_{j}}<t\}-\frac{1}{2}\textbf{1}\{\frac{X_{j}}{X_{i}}<t\}.
\end{gather*}

Let $X, Y$ be independent rv's with standard exponential distribution. The projection of this kernel $\xi(s;t)$ for fixed
$t \geq 0$ has the form:
\begin{gather*}
\xi(s; t) := E(\Xi(X, Y; t)\mid X=s) =\frac{t}{1+t}-\frac12\P(\frac{s}{Y}< t)-\frac12\P(\frac{Y}{s}< t).
\end{gather*}

After simple calculations we get the expression for the family of projections:
\begin{align}
\xi(s; t) =\frac{t}{1+t}-\frac12 e^{-\frac{s}{t}}+\frac12e^{-st}-\frac12 .\label{xi}
\end{align}

It is easy to see that $E(\xi (X; t))=0$.
The variance of this projection $\delta^2(t) = E\xi^2(X; t)$ under $H_{0}$
is  given by
\begin{equation*}
\delta^2(t) = \frac{t(t-1)^2(t^2+3t+1)}{4(t+1)^2(t+2)(t^3+(t+1)^3)}.
\end{equation*}
Hence,
\begin{equation*}
\delta^2=\sup_{ t \geq 0} \delta^2(t)\approx 0.00954.
\end{equation*}
This value will be important in the sequel when calculating the large deviation asymptotics.

\begin{figure}[h!]
\begin{center}
\includegraphics[scale=0.35]{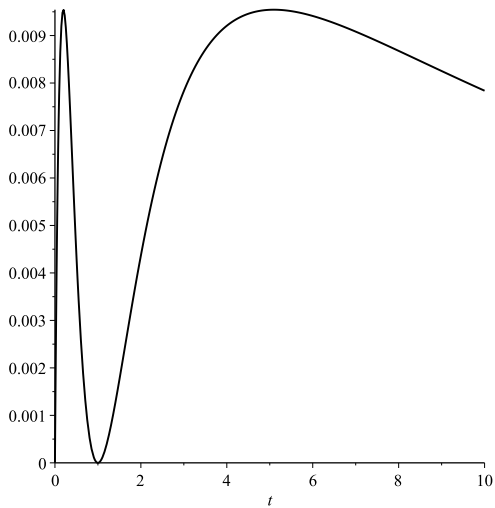}\caption{Plot of the function $\delta^2(t)$ }
\end{center}
\end{figure}

The limiting distribution of the statistic $D_n$ is unknown. Using the methods of \cite{Silv}, one can show that the
$U$-empirical process
$$\eta_n(t) =\sqrt{n} \left(\frac{t}{1+t} - H_n(t) \right), \ t\geq 0,
$$
weakly converges in $D(0,\infty)$ as $n \to \infty$ to certain centered Gaussian
process $\eta(t)$ with calculable covariance. Then the sequence of statistics
$\sqrt{n} D_n$ converges in distribution to  $\sup_{t\geq 0} |\eta(t)|.$ Currently we are not able to find explicitly its distribution.
Hence it is reasonable to determine the critical values for statistics  $D_n $ by simulation.

Table \ref{fig: tab6} shows the critical values of the null distribution of $D_n$ for significance
levels $\alpha = 0.1, 0.05, 0.01$ and specific sample sizes $n.$ Each entry is obtained by using the Monte-Carlo simulation
methods with 10,000 replications.

\begin{table}[htbp]
\centering
\bigskip
\caption{Critical values for the statistic $D_n$. }
\bigskip
\centering
\begin{tabular}{c|ccc}

$n$ &  0.1 &  0.05 &  0.01  \\
  \hline
   10    & 0.14  & 0.16 & 0.20 \\
  20 & 0.09  & 0.10 & 0.13  \\
   30    & 0.07  & 0.08 & 0.10  \\
  40     & 0.06  & 0.07 & 0.09 \\
   50   & 0.05  & 0.06 & 0.07 \\
   100     & 0.04 & 0.04 & 0.05\\ \hline
\end{tabular}
\label{fig: tab6}
\end{table}

Now we obtain the logarithmic large deviation asymptotics of the
sequence of statistics  $D_n$ under $H_0.$
The family of kernels $\{\Xi(X, Y; t), t\geq 0\}$ is not only centered but also bounded. Using the results from \cite{Niki10} on
large deviations for the supremum  of nondegenerate $U$-statistics, we obtain the following result.
\begin{theorem}
For  $a>0$ under $H_0$
$$
\lim_{n\to \infty} n^{-1} \ln \P (  D_n >a) = - f_D(a),
$$
where the function  $f_D$ is continuous for sufficiently small $a>0,$ moreover $$
f_D(a) = \frac{a^2}{8 \delta^2} (1 + o(1))  \sim 13.103 a^2, \, \mbox{as}
\, \, a \to 0.
$$
\end{theorem}

\subsection{Local Bahadur efficiency of $D_n$}

\begin{lem}\label{b_D}
For a given alternative density $g(x,\theta)$ from the class $\cal G$ (see Notation~\ref{nt}) we have
\begin{gather*}
b_D(t,\theta) \sim 2\theta  \int_{0}^{\infty} \xi(x; t)h(x)dx, \text{ where }
h(x)=g'_{\theta}(x,0).
\end{gather*}
\end{lem}
\begin{proof}
 By the Glivenko-Cantelli theorem
for $U$-statistics \cite{Jan} the limit in probability under the alternative for statistics $D_n$ is equal to
\begin{gather*}
b_D(t, \theta)=\frac{t}{1+t} - \P_{\theta}(\frac{X}{Y}< t).
\end{gather*}
Then as $\theta \rightarrow 0$
\begin{gather*}
b_D(t,\theta)=\frac{t}{1+t}
-\int_{0}^{\infty}g(y,\theta) dy \int_{0}^{yt}g(x,\theta) dx \sim b_D(t,0)+ b_D '(t,0)\cdot \theta.
\end{gather*}

It is easily seen that for $f(x) = e^{-x}, \ x \ge 0,$
\begin{eqnarray*}
J(0)&=&0, \\
J'(0)&=&-\int_{0}^{\infty}h(y) dy \int_{0}^{yt} f(x) dx-\int_{0}^{\infty}f(y) dy \int_{0}^{yt} h(x) dx.
\end{eqnarray*}
Changing the order of integration in the second integral we see that:
\begin{gather*}
J'(0)= -\int_{0}^{\infty}F(y t)h(y) dy
-\int_{0}^{\infty}h(x) dx \int_{x/t}^{\infty}f(y) dy =-\int_{0}^{\infty}h(y)\left(F(y t)+1-F(y/t)\right)dy=\\
=2 \int_{0}^{\infty} \xi(x; t)h(x)dx.
\end{gather*}
\end{proof}

Therefore
\begin{gather*}
b_D(\theta):= \sup_{t \geq 0}|b_D(t, \theta)| \sim \sup_{t \geq 0}| 2\theta  \int_{0}^{\infty} \xi(x; t)h(x)dx|.
\end{gather*}
We can find the asymptotics of $b_D(\theta)$ for each alternative as $\theta \to 0$:
\begin{gather*}
b_D^1(\theta):= \sup_{t \geq 0}|-\frac{t\ln(t)}{(t+1)^2}\theta| \sim 0.2239 \theta,\\
b_D^2(\theta):= \sup_{t \geq 0}|\frac{(t-1)\ln(t+1)-t\ln(t)}{(t+1)}\theta| \sim 0.1468 \theta,\\
b_D^3(\theta):= \sup_{t \geq 0}|\frac{(\beta-1)^2t(t-1)}{(t+1)(\beta+t)(t\beta+1)}\theta| = \sup_{t \geq 0}|\frac{4t(1-t)}{(t+1)(t+3)(3t+1))}\theta| =\\ = 0.1056 \cdot\theta \, \text{under} \, \beta =3,\\
b_D^4(\theta) \sim 0.1406 \theta.
\end{gather*}

\begin{figure}[h!]
\begin{minipage}[t]{0.33\textwidth}
\includegraphics[scale=0.21]{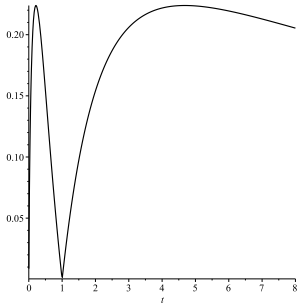}\caption{Plot of the func\-tion $b_D^1(t, \theta)$ }
\end{minipage}
\begin{minipage}[t]{0.33\textwidth}
\includegraphics[scale=0.21]{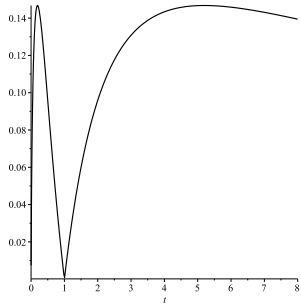}\caption{Plot of the function $b_D^2(t, \theta)$ }
\end{minipage}
\begin{minipage}[t]{0.32\textwidth}
\includegraphics[scale=0.21]{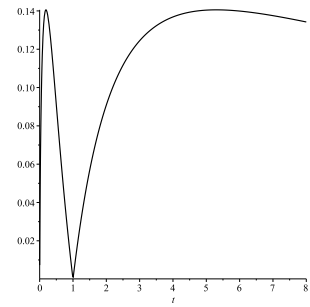}\caption{Plot of the function $b_D^4(t, \theta)$ }
\end{minipage}
\end{figure}

We cannot find the explicit formula for $b_D^4(t, \theta),$ and are forced to evaluate the maximal value of the $b_D^4(\theta)$ by using the numerical methods with MAPLE package. Again using (5) we present in Table 4 the values of exact slopes  when $\theta \to 0$ and the local Bahadur efficiencies against our four alternatives for statistics $D_n.$
\begin{table}[!hhh]\centering
\caption{Local Bahadur efficiency for $D_n$.}
\bigskip
\begin{tabular}{|c|c|c|c|}
\hline
Alternative &   $c_D(\theta)$& Efficiency\\
\hline
Weibull &    1.313 $\theta^2$  & 0.798 \\
 Gamma &    0.564  $\theta^2$  & 0.875  \\
EMNW ($\beta =3$)   & 0.292  $\theta^2$ &  0.821\\
Verhulst&0.518 $\theta^2$&0.886\\
\hline
\end{tabular}
\label{fig: tab3}
\end{table}

We see that the efficiency is reasonably high in all four examples. Moreover, it is {\it much higher} than usual values of efficiency for Kolmogorov test. In Table \ref{fig: tabsimd} we present the simulated powers for our four alternatives.
Again the simulations have been performed for $n=100$ with 10,000 replicates.

\begin{table}[htbp]
\centering
\bigskip
\caption{Simulated powers for statistic $D_n$. }
\bigskip
\centering
\begin{tabular}{|c|c|ccc|}
  \hline
 Alternative & $\theta$ &  $\alpha=0.1$ & $\alpha=0.05$ & $\alpha=0.01$  \\
  \hline
 Weibull&0.5 &  0.999 &0.997 &0.976  \\
   &0.25 & 0.809 & 0.712 & 0.452  \\
 \hline
  Gamma&0.5 & 0.914 & 0.845 &0.622   \\
   &0.25  & 0.489 & 0.361& 0.155  \\
 \hline
 EMNV($\beta=3$)&0.5 & 0.996 & 0.991 & 0.941  \\
   &0.25 & 0.504 & 0.382 & 0.171   \\
 \hline
Verhulst &0.5 & 0.883 & 0.797 & 0.552  \\
   &0.25 &  0.454 & 0.330 & 0.136   \\
 \hline
 \end{tabular}
\label{fig: tabsimd}
\end{table}

\section{Application to real data}

\indent In this section we apply our tests to an interesting real data example. We examine the data on the lengths of rule for Western Roman
Emperors by chronology of Kienast \cite{Kh} as the most precise. We consider two periods of this chronology: the period historians call "decline and fall",
taken conditionally from Nerva (reign: 96 - 98 AD) to Theodosius I (reign: 379 - 395 AD) with $n=53$ and
full period in data extending back to the first Roman Emperor, Augustus (reign: 27 BC - 14 AD), to Theodosius I with $n=76.$
The chronology shows the dates of ascent and abdication (or death). In the cases where exist no specific day of month we select a mid-points as it was done in \cite{Khm} and \cite{Khmal}.

In these papers  the authors came to the surprising agreement of data with the exponential distribution.
However, they used {\it only one} test for exponentiality of Kolmogorov type proposed by Haywood and Khmaladze \cite{HK}, \cite{Khmal}, \cite{Khm}. This test is probably not sensitive enough, and the single agreement with the exponentiality stated by these authors does not convince us in the validity of $H_0.$  First evidence that exponentiality  fails for this data  appeared in \cite{Pisk}.

It is interesting to apply our new tests of exponentiality to this challenging problem. We were based on 10,000 simulations of exponential data and calculated the $p$-values of new statistics $W_n$ with $\mu =2$ and $D_n$ understanding them as the probability, under the assumption of exponentiality hypothesis $H_0$, of obtaining a result equal to or more extreme than what was actually observed when calculating the test statistics.
 We got that for $n=76$ all $p$-values are less than $10^{-4},$ and in case of $n=53$
we got the following $p$-values:
\begin{center}
 \begin{tabular}{c|cc|cc}
 &   \multicolumn{2}{|c|}{$n=53$ }  &  \multicolumn{2}{|c}{$n=76$ } \\
  \hline
  test &$|W_{n}|, \mu = 2$  &$D_{n}$ & $|W_{n}|, \mu = 2$  &$D_{n}$ \\
  \hline
 value&  0.048 & 0.095 &  0.050 & 0.096 \\
 $p$-value&  0.0011 & 0.0019 & $< 10^{-4}$ & $< 10^{-4}$ \\
\end{tabular}
\end{center}

\noindent The smaller is the p-value, the more compelling is the evidence that {\it the alternative} should be accepted. Therefore we conclude that all our tests strongly reject the exponentiality of this data with the attained significance  level less than $\alpha=0.002$.

In the recent paper by El-Barmi and McKeague \cite{Keague} the authors used for the  sample of durations of reigns of Roman Emperors their new test on the ordering of distributions of several independent samples. Let say that the  rv
$X_1$ with df $F_1$ is stochastically larger than the rv $X_2$ with df $F_2$, if $F_1(x) \geq F_2(x) $ for all $x.$ It is denoted as  $F_1 \succ F_2.$ The authors of  \cite{Keague} supposed that the three periods of history of Roman Empire which usually are called the Principate (27 BC - 235 AD),
the Crisis of III century (235 AD - 284 AD) and the Dominate(285 AD - 395 AD) consists of independent but non-identically distributed periods of reign with df's $F_1, F_2$ and  $F_3$, which are probably non-exponential. Actually their test witnesses in favor of the hypothesis: Dominate $\succ$ Principate $\succ$ Crisis, or equivalently $ F_3 \succ F_1 \succ F_2$  and most probably does not support the hypothesis of exponentiality, too.

We  tried also other tests of exponentiality. The hypothesis is steadily rejected for the full sample of 76 Emperors in virtue of the Moran \cite{Mor}, chi-square,  Gini and Lilliefors \cite{NTch} tests. In the case of smaller sample $n=53,$ which corresponds to the "decline and fall,"\ of the Roman Empire, the agreement with exponentiality appears more often. However, the Moran test still rejects the exponentiality, and our two tests proposed above also append their contribution to rejection of the hypothesis under discussion.

\section{Conditions of local asymptotic optimality} \label{localop}

In this section we are interested in conditions of local asymptotic optimality (LAO) in Bahadur sense
for both sequences of statistics $W_n$  and $D_n.$ This means to describe the local structure of the
alternatives for which the given statistic has maximal potential local efficiency so that the relation
$$
c_T(\theta) \sim 2 KL(\theta),\,\, \mbox{as} \, \, \theta \to 0,
$$
holds (see \cite{Bahadur}, \cite{Nik}, \cite{NikTchir}, \cite{NiPe}).
Such alternatives form the so-called domain of LAO for the given sequence of statistics $\{T_n\}$.

Let again consider the densities $ g(\cdot \ ,\theta)$ with the df's $G(\cdot \ ,\theta)$ from the class $\cal G$ (see Notation~\ref{nt}) . Define the functions
\begin{gather*}
H(x)=\frac{\partial}{\partial\theta}G(x,\theta)\mid
_{\theta=0},\quad
h(x)=\frac{\partial}{\partial\theta}g(x,\theta)\mid _{\theta=0}.
\end{gather*}

\noindent Suppose also that for $G$ from $\cal G$ the following regularity conditions hold:
\begin{gather*}
h(x)=H'(x), \,  x \geq 0, \quad \int_0^\infty h^2(x)e^{x}dx <  \infty  , \\
\frac{\partial}{\partial\theta}\int_0^\infty x g(x,\theta)dx \mid
_{\theta=0} \ = \ \int_0^\infty x h(x)dx.
\end{gather*}

It is easy to show, see also \cite{NikTchir}, that under these conditions
$$
2KL(\theta)\sim \bigg[\int_0^\infty h^2(x)e^{x}dx -\Big(\int_0^\infty x h(x)dx\Big)^2\bigg] \theta^2,\,\, \mbox{as} \, \,\theta \to 0.
$$

\subsection{LAO conditions for $W_n$}

Now consider the integral statistic  $W_n$  with the kernel  $\Phi(X,Y)$ and its projection $\phi_{\mu}(x)$ from \eqref{phi}.
Let us introduce the auxiliary function
\begin{equation}
\label{h0}
h_0(x) = h(x) - (x-1)\exp(-x)\int_0^\infty u h(u) du.
\end{equation}

Simple calculations show that
\begin{gather}\label{LaoD}
\int_0^{\infty} h^2(x)e^{x}dx -\Big(\int_0^{\infty} x h(x)dx\Big)^2 = \int_0^{\infty} h_0^2(x) e^{x} dx,\\
\int_{0}^{\infty} \phi_{\mu}(x)h(x)dx = \int_{0}^{\infty} \phi_{\mu}(x)h_0(x)dx.
\end{gather}

Using the asymptotics from Lemma~\ref{b_W}, we get that the local BE takes the form
\begin{eqnarray*}
e^B(W)&=& \lim_{\theta \to 0} \frac{\big(b_W(\theta)\big)^2}{ 8 \Delta_W^2 (\mu) KL(\theta)}=\\
&= &\Big( \int_{0}^{\infty} \phi_{\mu} (x)h_0(x)dx\Big)^2/\Big(
\int_{0}^{\infty}\phi_{\mu}^2(x) e^{-x}dx \cdot  \int_0^\infty h_0^2(x)e^{x}dx
 \Big).
\end{eqnarray*}

Therefore the distributions with   $h(x) = e^{-x}(C_1\phi_{\mu}(x)+
C_2(x-1))$ for some constants $C_1>0$ and $C_2$ form the LAO domain in the class $\cal G$.

The simplest example of
such alternative density is $$g(x,\theta)=e^{-x}\big[1+\theta
(1- \mu e^{\mu}E_1 (\mu)- \sqrt{\mu x}\,K_1(\sqrt{\mu x})- \frac{x}{2(\mu+x)})\big]$$  for small $\theta > 0$.

\subsection{LAO conditions for $D_n$}

Now let us consider the Kolmogorov type statistic  $D_n$  with the family of kernels  $\Xi$ and their
projections $\xi(x;t)$ from \eqref{xi}. After simple calculations we get
\begin{gather*}
\int_{0}^{\infty} \xi(x; t)h(x)dx = \int_{0}^{\infty} \xi(x; t )h_0(x)dx, \quad  \forall t \in [0,\infty).
\end{gather*}

For $h_{0}(x)$ defined in \eqref{h0}, using the asymptotics for  $ b_D(t,\theta)$ from Lemma~\ref{b_D}
and from \eqref{LaoD}, the efficiency is equal to
\begin{eqnarray*}
e^B (D)&=& \lim_{\theta \to 0}\frac{ \big(b_{D}(\theta)\big)^{2}} {\sup_{t\geq 0}\big(8
\delta^2(t)\big) KL(\theta) }\\
&=& \sup_{t\geq 0}\Big( \int_{0}^{\infty}
 \xi (x;t)h_0(x)dx\Big)^2 / \ \sup_{t\geq 0} \Big(
\int_{0}^{\infty}\xi^2 (x;t) e^{-x}dx \cdot \int_0^\infty h_0^2 e^{x} dx\Big).
\end{eqnarray*}

 From Cauchy-Schwarz inequality we obtain that efficiency is equal to 1 if
$ h(x)=e^{-x}\big(C_3\xi(x; t_0)+ C_4(x-1)\big)$
for $t_0= {\rm argmax}_{t\geq0}
\delta^2(t) $ and some constants $C_3>0$ and $C_4.$
The alternative densities having such function $h(x)$ form the domain of LAO in the
corresponding class. The simplest example is 
$$g(x,\theta)=e^{-x}\big[1+\theta \cdot
(\frac{t_0}{1+t_0}-\frac12 e^{-\frac{x}{t_0}}+\frac12e^{-xt_0}-\frac12)\big],$$ where
$$t_0= \underset{t\geq0}{\rm argmax} \Big(
 \frac{t(t-1)(t^3+2t^2-2t-1)}{4(t+1)^2(t+2)(t^3+(t+1)^3)}\Big)\approx \left[ \begin{aligned}0.1963 \\ 5.0949.\end{aligned}\right.$$

\section{Conclusion}

We have proposed in this paper two new tests of exponentiality which use a particular property of the exponential law but are not consistent against {\it any alternative.} In the same time they are rather sensitive against the deviations from exponentiality. This is sustained by their high local Bahadur efficiency and considerable power under common alternatives. Our tests were able to reject the exponentiality of
the sample of reigns of Roman emperors which was claimed by Khmaladze and his coauthors in \cite{HK}, \cite{Khmal}, \cite{Khm}. We hope that our tests will be useful in other delicate cases when one has to confirm the rejection of exponentiality hypothesis. Finally we have described the structure of "most favorable" alternatives to exponentiality under which our tests become locally optimal in Bahadur sense.

\section{Acknowledgements}

Research supported by grant RFBR No. 16-01-00258.

\end{document}